\documentclass{amsart}
\usepackage{graphicx}
\usepackage{hyperref}
\usepackage{enumitem}
\usepackage{theoremref}
\usepackage{float}
\usepackage{physics}

\DeclareMathOperator{\co}{Co}                               
\DeclareMathOperator{\diag}{diag}

\newtheorem{theorem}{Theorem}[section]
\newtheorem{proposition}[theorem]{Proposition}
\newtheorem{lemma}[theorem]{Lemma}
\newtheorem{corollary}[theorem]{Corollary}

\theoremstyle{definition}           

\newtheorem{remark}[theorem]{Remark}
\newtheorem{example}[theorem]{Example}

\begin{document}
\title{Matrices whose field of values is inscribed in a polygon}\thanks{Supported by NSF Award \href{https://www.nsf.gov/awardsearch/showAward?AWD_ID=2150511}{DMS-2150511}.}

\author[M.~J.~Fyfe]{Matthew J.~Fyfe}
\address{Bowling Green State University, Bowling Green, OH, U.S.A.}

\author[Y.~Hernandez]{Yesenia Hernandez}
\address{Bryn Mawr College, Bryan Mawr, PA, U.S.A.}

\author[P.~Paparella]{Pietro Paparella}
\address{Division of Engineering \& Mathematics, University of Washington Bothell, Bothell, WA, U.S.A.}
\email{pietrop@uw.edu}

\author[M.~Rajbhandari]{Malini Rajbhandari}
\address{Bryn Mawr College, Bryan Mawr, PA, U.S.A.}



\begin{abstract} 
In this work, it is shown that if $A$ is an $n$-by-$n$ \emph{convexoid} matrix (i.e., its field of values coincides with the convex hull of its eigenvalues), then the field of any $(n-1)$-by-$(n-1)$ principal submatrix of $A$ is inscribed in the field of $A$, i.e., the field is tangent to every side of the polygon corresondping to the boundary of the field of $A$. This result generalizes a special case established by Johnson and Paparella [Amer.~Math.~Monthly 127 (2020), no.~1,45–53]. 
\end{abstract}

\maketitle

\section{Introduction}

The \emph{field of values} (or \emph{numerical range}) of a matrix $A$ is the image of the two-norm unit-sphere in complex Euclidean space with respect to the map $x \longmapsto x^\ast A x$.   

Recently, Johnson and Paparella \cite{jp2020} used various concepts from matrix analysis, including the \emph{discrete Fourier transform matrix}, the field of values, \emph{trace vectors}, and \emph{differentiators}, to provide a framework that admits short proofs of the Gauss--Lucas and B\^{o}cher--Grace--Marden theorems (the latter is often simply referred to as \emph{Marden's theorem}), which are classical results in the geometry of polynomials. 

In particular, and germane to what follows, Johnson and Paparella \cite[pp.~5--6]{jp2020} proved that if $A = FDF^\ast$, where $D = \diag{(\lambda_1,\ldots,\lambda_n)}$ and $F$ is the $n$-by-$n$ discrete Fourier transform matrix, then the principal submatrix $F(A_{(1)})$, obtained by deleting the first-row and first column of $A$, is tangent to the midpoints of every side of the polygon $\partial F(A) = \partial \co{(\lambda_1,\ldots,\lambda_n)}$. 

In this work, this result is generalized to the fullest extent possible---in particular, it is shown that if $A$ is \emph{convexoid}, i.e., $F(A)$ coincides with the convex hull of its eigenvalues, then $F(A_{(k)})$ is inscribed in the polygon $\partial F(A)$, $\forall k \in \{1,\ldots,n\}$.  

\section{Notation and Background}

The set of $m$-by-$n$ matrices with entries over $\mathbb{C}$ is denoted by $\mathsf{M}_{m \times n}(\mathbb{C})$; when $m = n$, $\mathsf{M}_{n \times n}(\mathbb{C})$ is abbreviated to $\mathsf{M}_n$. The set of all $n$-by-$1$ column vectors is identified with the set of all ordered $n$-tuples with entries in $\mathbb{C}$ and thus denoted by $\mathbb{C}^n$. If $x \in \mathbb{C}^n$, then $x_i$ denotes the $i$th entry of $x$. The $n$-by-$n$ identity matrix is denoted by $I = I_n$ and $e_k$ denotes the $k$th column of $I$.

Given $A \in \mathsf{M}_n$, we let 
\begin{itemize}
    \item \( \sigma(A) \) denote the \emph{spectrum} (i.e., multiset of eigenvalues) of \(A\);
    \item \(A_{(k)}\) denote the \((n-1)\)-by-\((n-1)\) \emph{principal submatrix} obtained by deleting the \(k\){th} row and \(k\){th} column of \(A\)); and
    \item $A^\ast$ denotes the \emph{conjugate transpose of $A$}.
\end{itemize}

If $A \in \mathsf{M}_n$ and $B \in \mathsf{M}_m$, then the \emph{direct sum of $A$ and $B$}, denoted by $A \oplus B$, is defined by 
\[ A \oplus B =\begin{bmatrix} A & 0 \\ 0 & B \end{bmatrix}. \]

If $U \in \mathsf{M}_n$, then $U$ is called \emph{unitary} if $U^* U = I$. If $A \in \mathsf{M}_n$, then $A$ is called \emph{normal} if $A^* A = A A^*$. A matrix $A$ is {normal} if and only if there is a unitary matrix $U$ and a diagonal matrix $D$ such that $A = U D U^*$ \cite[Theorem 2.5.3(b)]{hj2013}.

The \emph{field (of values)} or \emph{numerical range of $A \in \mathsf{M}_n)$}, denoted by $F(A)$, is defined by \( F(A) = \left\{ x^*A x \mid x^*x = 1 \right\} \subseteq \mathbb{C} \). A general reference for the field is \cite[Chapter 1]{hj1994}.

If $S = \{ \lambda_1,\ldots,\lambda_n \} \subset \mathbb{C}$ (repetitions allowed), then the \emph{convex hull} of $S$ is denoted by $\co S = \co{(S)}$.

The following well-known properties will be useful in the sequel: 

\begin{proposition} 
    \thlabel{fvprops}
        If $A = [a_{ij}] \in \mathsf{M}_n$ and $B \in \mathsf{M}_m$, then:
     \begin{enumerate}
        [label=(\roman*)]
            \item \label{nrspectrum} \( \sigma(A) \subseteq F(A) \) \cite[Property 1.2.6]{hj1994}; 
            \item \label{nrnormal} \( F(A) = \co{(\sigma(A))} \), whenever \(A\) is {normal} \cite[Property 1.2.9]{hj1994}; 
            \item \thlabel{ds} $F ( A\oplus B ) = \co(F(A) \cup F(B))$ \cite[Property 1.2.10]{hj1994};  
            \item \label{nrsubmatrix} \( F(A_{(k)}) \subseteq F(A) \), $\forall  k \in \{1,\ldots, n \}$ \cite[Property 1.2.11]{hj1994}; and
            \item \label{toephaus}\(F(A)\) is convex \cite[\S 1.3]{hj1994}.
    \end{enumerate}
\end{proposition}

\begin{proof}
For completeness, we give a proof of Property 1.2.11 \cite[p.~13]{hj1994}, which generalizes part \ref{nrsubmatrix}, given that it is ubiquitous in the literature; the proof-strategy suggested by Horn and Johnson is tedious; and ideas presented in the demonstration will be used in the sequel. 

To this end, let $\alpha = \{ \alpha_1, \dots, \alpha_m \}$ be a nonempty subset of $\{1, \dots, n\}$ (if $\alpha = \emptyset$, then $F(A[\alpha]) = \emptyset \subseteq F(A)$) and denote by $A[\alpha]$ the $m$-by-$m$ matrix whose $(i,j)$ entry is $a_{\alpha_i, \alpha_j}$, $1 \le i,j \le m$. If $P:= \begin{bmatrix} e_{\alpha_1} & \cdots & e_{\alpha_m}\end{bmatrix}$, then
\begin{align*}
    a_{\alpha_i, \alpha_j}
    = e_{\alpha_i}^\top A e_{\alpha_j}
    = \left[ P^\top A P\right]_{ij},
\end{align*}
i.e., $A[\alpha] = P^\top A P$.

If $z \in F(A[\alpha])$, then $z = x^*F(A[\alpha])x$, where $x^* x = 1$. If $y:=Px$, then $y^*y=x^*P^\top Px = x^*I_m x= x^*x=1$. Furthermore,
\[z = x^* A[\alpha] x = x^* P^\top A P x=(Px)^*APx = y^*Ay \in F(A), \] 
i.e, $F(A[\alpha]) \subseteq F(A)$.
\end{proof}

If $A \in \mathsf{M}_n(\mathbb{C})$, then $A$ is called \emph{convexoid} if $F(A) = \co(\sigma(A))$. Johnson \cite[Theorem 3]{j1976} gave the following characterization of convexoid matrices.

\begin{theorem}
    \thlabel{convexoid}
        If $A \in \mathsf{M}_n(\mathbb{C})$, then $A$ is convexoid if and only if $A$ is normal or there is a unitary matrix $U$ such that 
        \[ U^\ast A U = \begin{bmatrix} A_1 & 0 \\ 0 & A_2 \end{bmatrix},\]
        where $A_1$ is normal and $F(A_2) \subseteq F(A_1)$.
\end{theorem}

By \thref{fvprops}\eqref{nrnormal}, if $A$ is normal and $\lambda_1,\ldots,\lambda_n$ are its eigenvalues (repetitions included), then $F(A) = \co{(\lambda_1,\ldots,\lambda_n)}$. Without loss of generality, we may label the vertices as $\lambda_1,\ldots, \lambda_d$, $1 \le d \le n$. Notice that 
\[\partial F(A) = \bigcup_{k=1}^d \co{(\lambda_k,\lambda_{k+1})}, \] 
where, for convenience, $d + 1 := 1$. We say that \emph{$F(A_{(k)})$ is inscribed in $F(A)$} if $F(A_{(k)}) \cap \co{(\lambda_k,\lambda_{k+1})} \ne \emptyset$, $\forall k \in \{1,\ldots,d\}$. 

\section{Main Result}

\begin{lemma}
    \thlabel{rotevec}
        Let $Av= \lambda v$,  where $v \neq 0$ and $ v_k \neq 0$. If $v_k = r\exp(i \theta)$, where $\theta \in (-\pi,\pi]$, then $w:= \exp(-i\theta)v$ is an eigenvector such that $\begin{Vmatrix} v \end{Vmatrix}_2 = \begin{Vmatrix} w \end{Vmatrix}_2$ and $w_k > 0$.
\end{lemma}

\begin{proof}
    The conclusion that $w$ is an eigenvector with the same length as $v$ follows from the fact that $\vert \exp(-i\theta) \vert = 1$. Lastly, notice that $w_k = \exp(-i\theta) (r\exp(i \theta)) = r > 0$.
\end{proof}

\begin{lemma}
    \thlabel{projection}
        Let $\alpha = \{ \alpha_1, \dots, \alpha_m \}$ be a nonempty subset of $\{1, \dots, n\}$, let $P:= \begin{bmatrix} e_{\alpha_1} & \cdots & e_{\alpha_m}\end{bmatrix}$, and let $y \in \mathbb{C}^n$ be any vector such that $y_k  = 0$ whenever $k \notin \alpha$. If $x:= P^\top y \in \mathbb{C}^m$, then $y = Px$ and $y^\ast y = x^\ast x$.
\end{lemma}

\begin{proof}
    Since $PP^\top = \sum_{k=1}^m e_{\alpha_k}e_{\alpha_k}^\top$, it follows that 
    \[ Px = P(P^\top y) = (P P^\top) y = \left( \sum_{k=1}^m e_{\alpha_k}e_{\alpha_k}^\top \right) y = \sum_{k=1}^m e_{\alpha_k} y_{\alpha_k}. \]
    If $z := \sum_{k=1}^m e_{\alpha_k}$, then $z_k = 0$ whenever $k \notin \alpha$. Thus, $y = Px$ and 
    \[ y^\ast y = (Px)^\ast (Px) = x^\ast P^\top P x = x^\ast I_m x = x^\ast x, \]
    as desired.
\end{proof}

\begin{theorem}
    \thlabel{main}
        If $A$ is convexoid, then $F(A_{(k)})$ is inscribed in the polygon $\partial F(A)$, $\forall k \in \{1,\ldots,n\}$.
\end{theorem}

\begin{proof}
In view of \thref{fvprops}\ref{ds} and \thref{convexoid}, it suffices to consider the case when $A$ is normal.

To this end, if $A$ is normal, then there is a diagonal matrix $D = \diag(\lambda_1,\dots,\lambda_n)$ and a unitary matrix $U$ such that $A = UDU^*$. 

Let $\lambda_i$ and $\lambda_j$ be adjacent vertices of the polygon $\partial \co(\lambda_1,\dots,\lambda_n)$, let $v := U e_i$, and let $w := Ue_j$. Note that $v^\ast v = w^\ast w = 1$ and $v^\ast w = w^\ast v = 0$. We distinguish the following cases:
\begin{enumerate}
\item $v_k = 0$ or $w_k =0$. If 
\begin{equation}
    \label{alphaindex}
    \gamma := \{1,\ldots,n \} \backslash \{k\} = \{ \gamma_1,\ldots, \gamma_{n-1} \},
\end{equation}
\begin{equation}
    \label{pmatrix}
    P := \begin{bmatrix} e_{\gamma_1} & \cdots & e_{\gamma_{n-1}} \end{bmatrix},
\end{equation} 
and $x := P^\top v$, then $v = Px$ and $x^\ast x = 1$ by \thref{projection}. Thus, 
\[ \lambda_i = v^\ast Av = (Px)^\ast A (Px) = x^\ast P^\top A P x = x^\ast A_{(k)} x \in F(A_{(k)}),  \]
i.e., $F(A_{(k)}) \cap \co{(\lambda_{i-1},\lambda_{i})} \ne \emptyset$ and $F(A_{(k)}) \cap \co{(\lambda_{i},\lambda_{i+1})} \ne \emptyset$ (if $i=1$, then $i-1 := d$). 

Similarly, if $w_k = 0$, then $\lambda_j \in F(A_{(k)})$. If $v_k = w_k = 0$, then the line-segment $\co{(\lambda_i,\lambda_j)} \subseteq F(A_{(k)})$ by \thref{fvprops}\ref{toephaus}. 
    
\item \textit{$v_k \ne 0$ and $w_k \ne 0$}. By \thref{rotevec}, it may be assumed, without loss of generality, that $v_k > 0$ and $w_k > 0$. If 
\begin{equation}
    \label{alphabeta}
        \alpha := \frac{\mp w_k}{\sqrt{v_k^2 + w_k^2}} \text{ and } \beta := \frac{\pm v_k}{\sqrt{v_k^2 + w_k^2}},
\end{equation}
then $\alpha$ and $\beta$ are nonzero reals such that \(\alpha^2 + \beta^2 = 1\). If $u := \alpha v + \beta w$, then 
\[ u_k = \alpha v_k + \beta w_k = \frac{\mp v_k w_k}{\sqrt{v_k^2 + w_k^2}} + \frac{\pm v_k w_k}{\sqrt{v_k^2 + w_k^2}} = 0 \]
and, since $v^\ast v = w^\ast w = 1$ and $v^\ast w = w^\ast v = 0$, it follows that
\[ u^\ast u = (\alpha v^\ast + \beta w^\ast)(\alpha v + \beta w) = \alpha^2v^\ast v + \alpha \beta v^\ast w + \beta \alpha w^\ast v + \beta^2 w^\ast w = \alpha^2 + \beta^2 = 1 \]
and 
\begin{align*} 
    u^\ast A u 
    &= (\alpha v^\ast + \beta w^\ast) (\alpha A v + \beta Aw)                                                   \\
    &= (\alpha v^* + \beta w^*)(\alpha \lambda_i v + \beta \lambda_j w)                                         \\
    &= \alpha^2 \lambda_i v^\ast v + \alpha \beta v^\ast w + \beta \alpha w^\ast v + \beta^2 \lambda_jw^\ast w  \\
    &= \alpha^2 \lambda_i + \beta^2 \lambda_j.
\end{align*}
Because $\alpha^2 + \beta^2 = 1$, $\alpha^2 > 0$, and $\beta^2 > 0$, it follows that $u^\ast A u$ lies in the interior of the line-segment $\co(\lambda_i,\lambda_j)$. Finally, if $x := P^\top u \in \mathbb{C}^{n-1}$, where $\gamma$ and $P$ are defined as in \eqref{alphaindex} and \eqref{pmatrix}, respectively, then $u = Px$ and $x^\ast x = 1$ (by \thref{projection}) and 
\[ \alpha^2 \lambda_i + \beta^2 \lambda_j = u^\ast A u = x^\ast P^\top A P x = x^\ast A_{(k)} x \in F(A_{(k)}). \]
Thus, $F(A_{(k)}) \cap \co{(\lambda_i,\lambda_j)} \ne \emptyset$.
\end{enumerate}
In all cases, $F(A_{(k)}) \cap \co{(\lambda_k,\lambda_{k+1})} \ne \emptyset$, $\forall k \in \{1,\ldots,d\}$.  
\end{proof}

\begin{remark}
Although there are two possible choices for $\alpha$ and $\beta$, the convex combination $\alpha^2 + \beta^2 = 1$ is unique. Therefore, in the case that $v_k \ne 0$ and $w_k \ne 0$, $F(A_{(k)})$ intersects the interior of the line segment $\co{(\lambda_i,\lambda_j)}$ at a single point.
\end{remark}

\begin{example}
If 
\[ U =
\begin{bmatrix}
    \frac{1}{2} & \frac{\sqrt{3}}{2}    & 0                     & 0                     \\
    \frac{1}{2} & -\frac{\sqrt{3}}{6}   & 0                     & \frac{\sqrt{6}}{3}    \\
    \frac{1}{2} & -\frac{\sqrt{3}}{6}   & \frac{\sqrt{2}}{2}    & -\frac{\sqrt{6}}{6}   \\
    \frac{1}{2} & -\frac{\sqrt{3}}{6}   & -\frac{\sqrt{2}}{2}    & -\frac{\sqrt{6}}{6}   \\
\end{bmatrix},
\]
then $U^\top U = I_4$. If $A = UDU^\top$, where $D = \diag{(-1-5i,-2,3-2i,2+5i)}$, then $A$ is normal. Figure \ref{fig:convexoid} illustrates \thref{main}. 

\begin{figure}[H]
\centering
\includegraphics[width=0.75\linewidth]{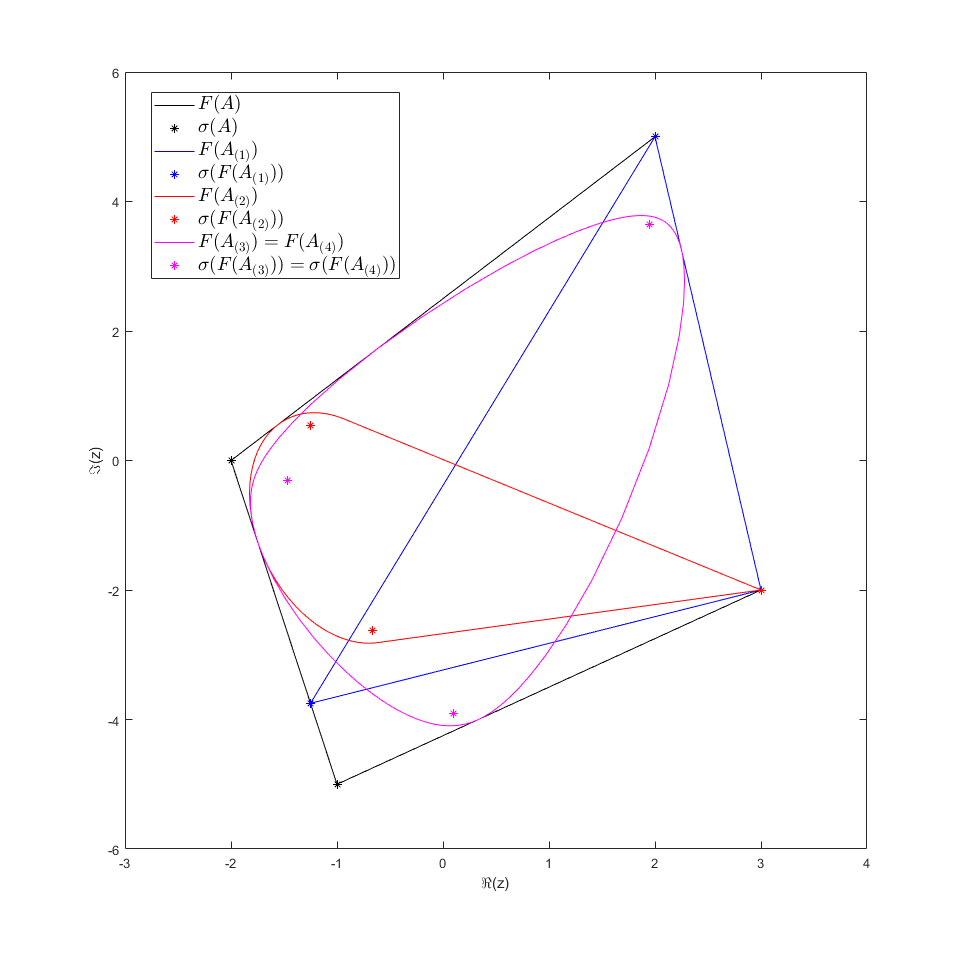}
\caption{An example generated via MATLAB illustrating \thref{main}.}
\label{fig:convexoid}
\end{figure}
\end{example}

\begin{corollary}
    \thlabel{dftcor}
        If $\lambda_1,\ldots,\lambda_n$ are complex numbers, $F$ denotes the $n$-by-$n$ discrete Fourier transform matrix, $D=\diag{(\lambda_1,\ldots,\lambda_n)}$, and $A = FDF^\ast$, then $F(A_{(k)})$ is inscribed in the polygon $\partial\co{(\lambda_1,\ldots,\lambda_n)}$ and the points of tangency occur at the midpoints of sides of $\partial\co{(\lambda_1,\ldots,\lambda_n)}$. 
\end{corollary}

\begin{proof}
    Since 
    \[ f_{ij} = \frac{\omega^{(i-1)(j-1)}}{\sqrt{n}}, \]
    where $\omega := \exp(-2 \pi i/n)$, it follows that 
    \[ \alpha = \frac{-\frac{1}{\sqrt{n}}}{\sqrt{\frac{1}{n} + \frac{1}{n}}} = \frac{-\frac{1}{\sqrt{n}}}{\frac{\sqrt{2}}{\sqrt{n}}} = -\frac{1}{\sqrt{2}}, \]
    \[ \beta = \frac{\frac{1}{\sqrt{n}}}{\sqrt{\frac{1}{n} + \frac{1}{n}}} = \frac{\frac{1}{\sqrt{n}}}{\frac{\sqrt{2}}{\sqrt{n}}} = \frac{1}{\sqrt{2}}, \]
    and $\alpha^2 = \beta^2 = \frac{1}{2}$, where $\alpha$ and $\beta$ are defined as in \eqref{alphabeta}. 
\end{proof}

\section*{Acknowledgements}

The authors thank the National Science Foundation for funding and the University of Washington Bothell for hosting \emph{REU Site: Tiling Theory, Knot Theory, Optimization, Matrix Analysis, and Image Reconstruction}. In addition, we thank Casey Mann and Milagros Loreto for their efforts.

\bibliographystyle{abbrv}
\bibliography{refs}

\end{document}